\documentclass[twoside,12pt,a4paper]{amsart} 
\usepackage[utf8]{inputenc}
\usepackage[english]{babel}

\usepackage{amsmath}
\usepackage{amsfonts}
\usepackage{amssymb}
\usepackage{color}

\usepackage{amsthm}

\newtheorem{lemma}{Lemma} [section]

\newtheorem{theorem}{Theorem} [section]

\newtheorem{open}{Open Question}

\usepackage{graphics}
\usepackage{epsfig}

\usepackage{url}
\usepackage{hyperref}

\usepackage[a4paper,top=2.5cm,bottom=2.8cm,
left=3.2cm,right=3.2cm]{geometry}
\markleft{\hfill \textsc{Relationships Between Six Incircles} \hfill}
\begin{document}
Sangaku Journal of Mathematics (SJM) \copyright SJM \\
ISSN 2534-9562 \\
Volume 3 (2019), pp.51-66  \\
Received 14 July 2019. Published on-line 04 August 2019 \\ 
web: \url{http://www.sangaku-journal.eu/} \\
\copyright The Author(s) This article is published 
with open access\footnote{This article is distributed under the terms of the Creative Commons Attribution License which permits any use, distribution, and reproduction in any medium, provided the original author(s) and the source are credited.}. \\
\bigskip
\bigskip

\begin{center}
{\Large \textbf{Relationships Between Six Incircles}} \\
\medskip
\bigskip
\textsc{Stanley Rabinowitz} \\
545 Elm St Unit 1,  Milford, New Hampshire 03055, USA \\
e-mail: \href{mailto:stan.rabinowitz@comcast.net}{stan.rabinowitz@comcast.net} \\
web: \url{http://www.StanleyRabinowitz.com/} \\
\end{center}
\bigskip

\textbf{Abstract.} If $P$ is a point inside $\triangle ABC$, then the cevians
through $P$ divide $\triangle ABC$ into six smaller triangles.
We give theorems about the relationship between the radii of the circles inscribed in these triangles.



\bigskip
\bigskip
\section{Introduction}

\newcommand{\degrees}{^\circ}

Japanese Mathematicians of the Edo period were fond of finding relationships
between the radii of circles associated with triangles.
For example, the 1781 book ``Seiyo Sampo" \cite[no.~84]{Fujita} gives the relationship
$r=\sqrt{r_1r_2}+\sqrt{r_2r_3}+\sqrt{r_3r_1}$ between the radii of the circles in Figure \ref{fig:1814} below.
This result was later inscribed on
an 1814 tablet in the Chiba prefecture \cite[p.~30]{Fukagawa-Pedoe}.

\bigskip
\begin{figure}[h!]
\centering
\includegraphics[width=0.5\linewidth]{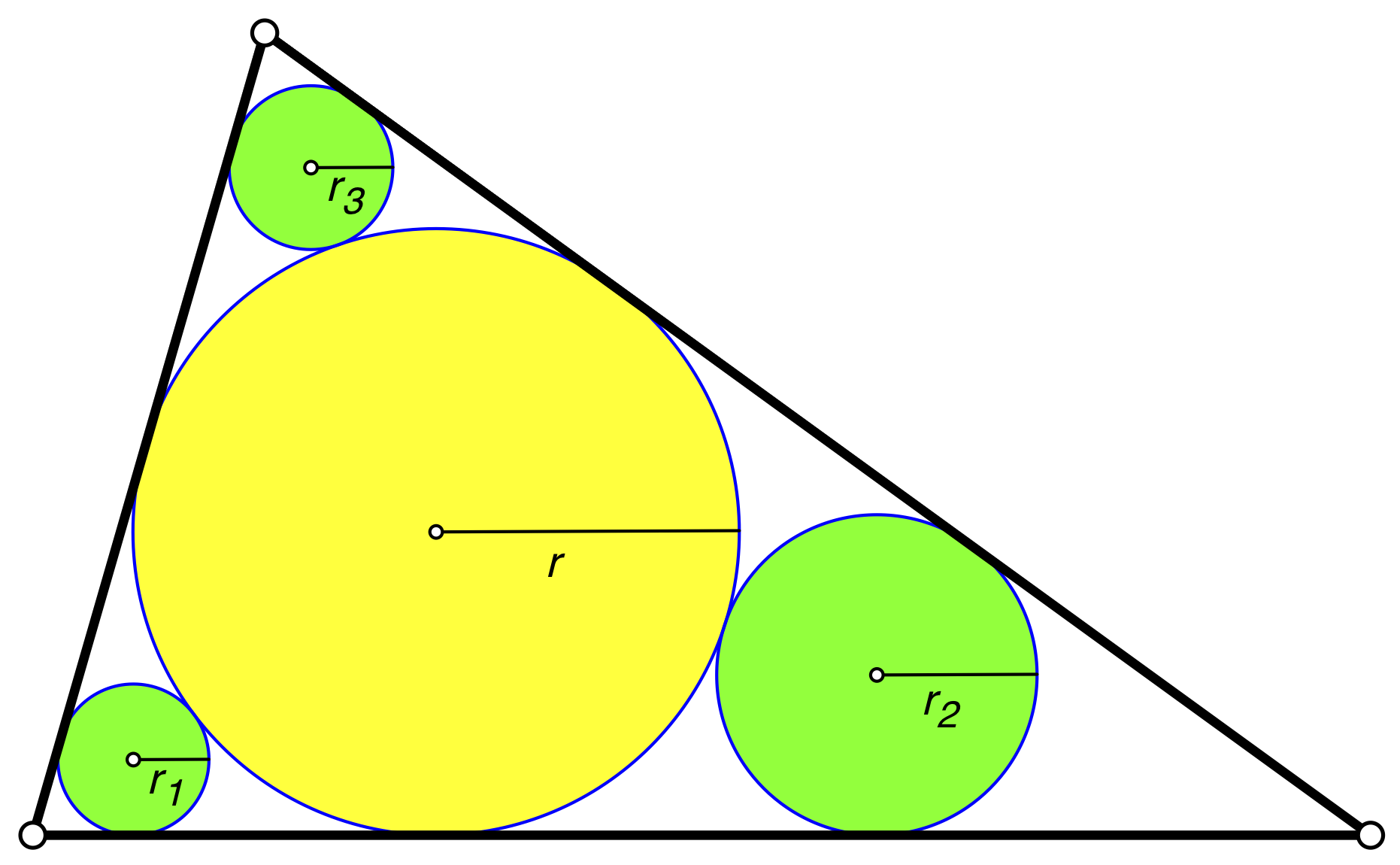}
\caption{}
\label{fig:1814}
\end{figure}

In the spirit of Wasan, we investigate the relationships between the radii of
six circles associated with a triangle and the three cevians through a point
inside that triangle.

\bigskip 
\section{Notation}

Let $P$ be any point inside a triangle $ABC$. The cevians
through $P$ divide $\triangle ABC$ into six smaller triangles.
Circles are inscribed in these six triangles.
The six triangles and their incircles are numbered from 1 to 6 counterclockwise as shown in Figure \ref{fig:preliminary}, where the first triangle is formed by the lines $AP$, $BP$, and $BC$.
The i-th triangle has inradius $r_i$,
semiperimeter $s_i$, circumradius $R_i$, and area $K_i$.

\bigskip
\begin{figure}[h!]
\centering
\includegraphics[width=0.43\linewidth]{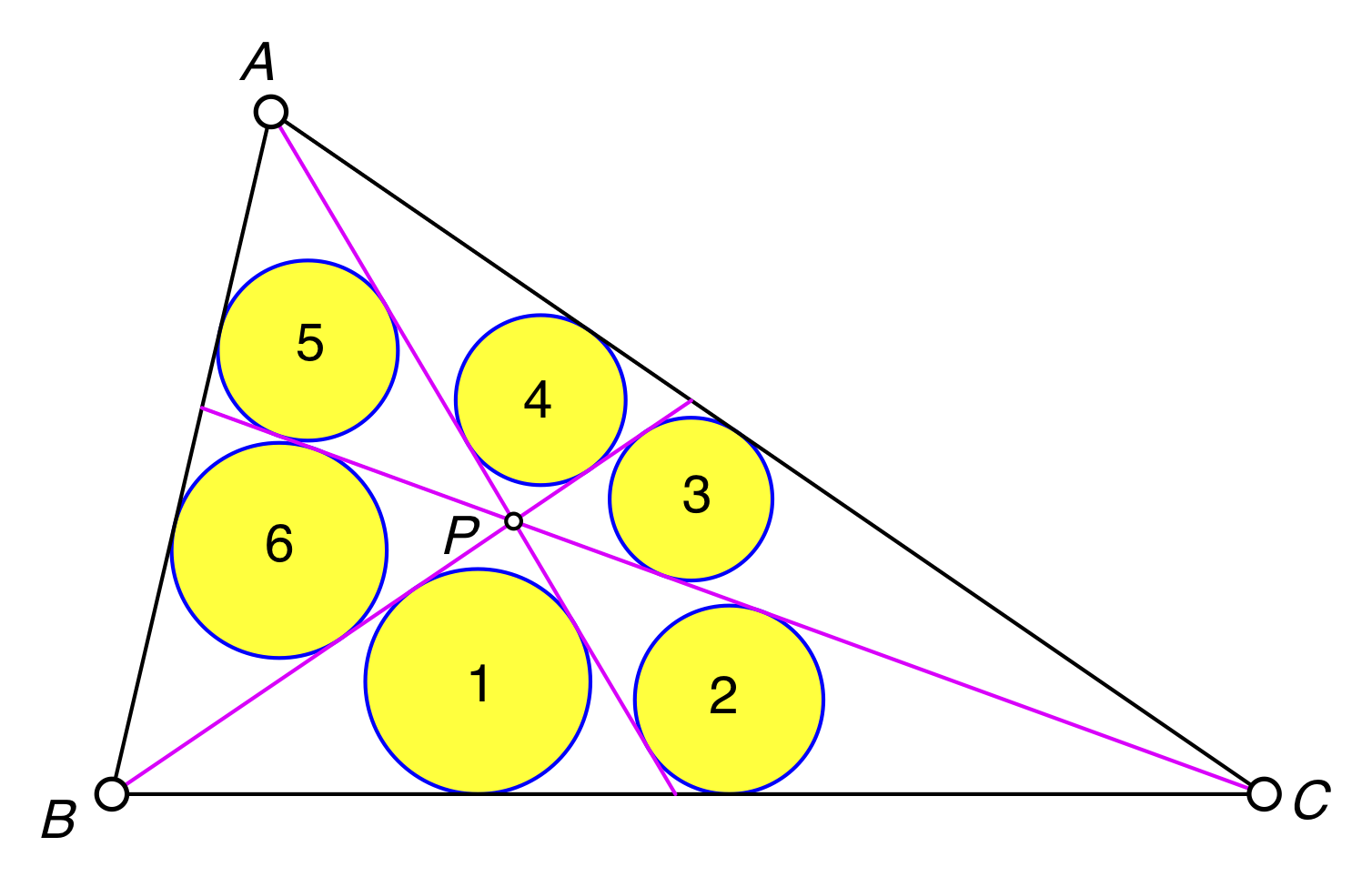}
\caption{numbering}
\label{fig:preliminary}
\end{figure}

If $X$ and $Y$ are points, then we use the notation $XY$ to denote either the line segment joining $X$ and $Y$ or the length of that line segment, depending on the context.

\bigskip 
\section{The Orthocenter}

We start with a known result \cite{Gutierrez} giving the relationship between the $r_i$ when $P$ is the orthocenter.

\begin{theorem}
\label{thm:Orthocenter}
If $P$ is the orthocenter of $\triangle ABC$ (Figure \ref{fig:orthocenter}), then $r_1r_3r_5=r_2r_4r_6$.
\end{theorem}

\bigskip
\begin{figure}[h!]
\centering
\includegraphics[width=0.5\linewidth]{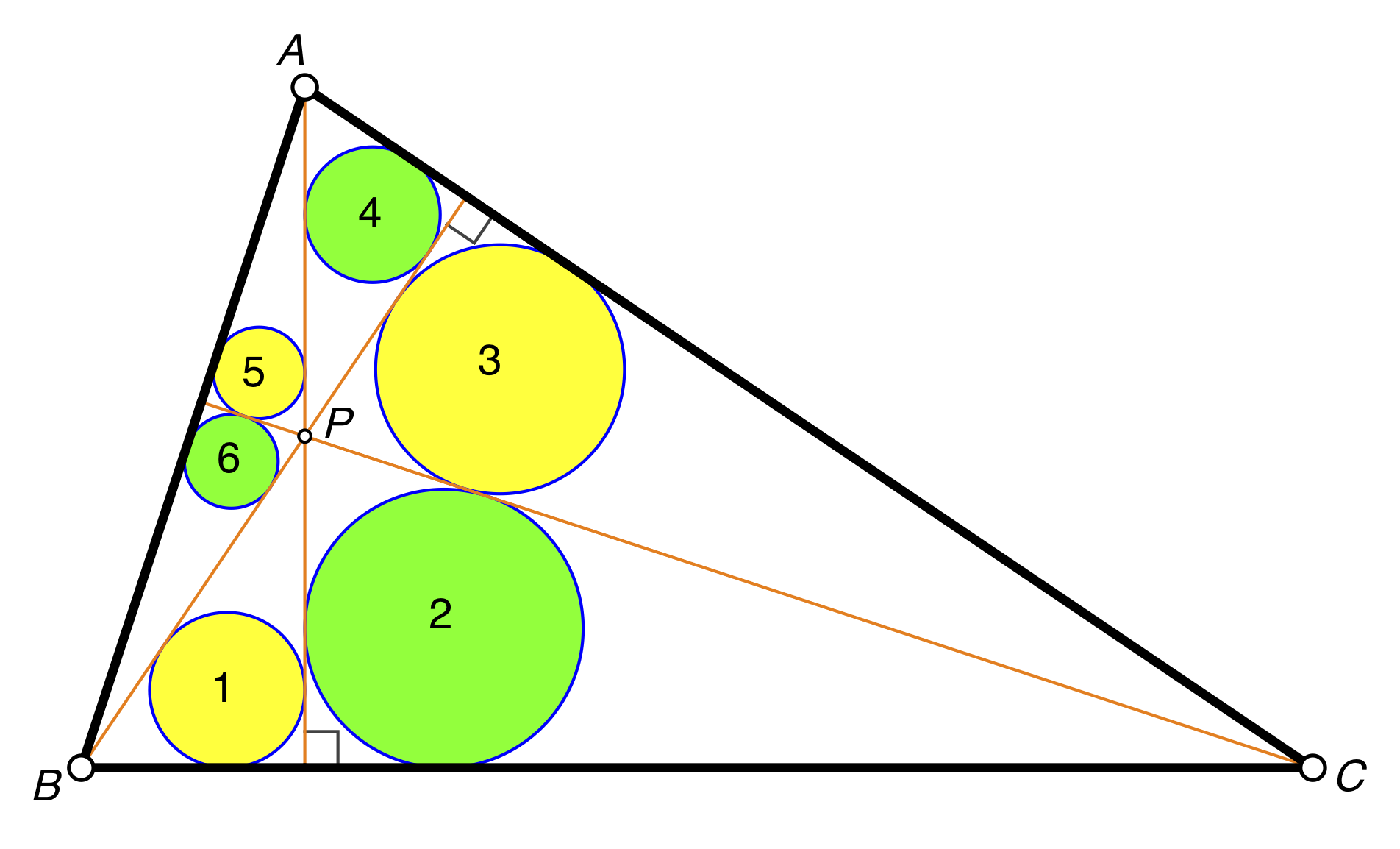}
\caption{$r_1r_3r_5=r_2r_4r_6$}
\label{fig:orthocenter}
\end{figure}

The following proof comes from \cite{hexram}.

\begin{proof}
Note that triangles 1 and 4 in Figure \ref{fig:orthocenter} are similar. Corresponding lengths in similar triangles are in proportion, so $r_1/BP=r_4/AP$. Triangles 2 and 5 are also similar as well as triangles 3 and 6, giving similar proportions. Therefore
$$\frac{r_1}{BP}\cdot\frac{r_3}{CP}\cdot\frac{r_5}{AP}
=\frac{r_4}{AP}\cdot\frac{r_6}{BP}\cdot\frac{r_2}{CP}$$
which implies
$r_1r_3r_5=r_2r_4r_6$.
\end{proof} 

\bigskip 
\section{The Centroid}

Next, we will consider the case when $P$ is the centroid. We start with a few lemmas.

\begin{lemma}
\label{lemma:medians}
If $P$ is the centroid of $\triangle ABC$ (Figure \ref{fig:medians}), then the six triangles formed all have the same area.  That is,
$$K_1=K_2=K_3=K_4=K_5=K_6.$$
\end{lemma}

\bigskip
\begin{figure}[h!]
\centering
\includegraphics[width=0.5\linewidth]{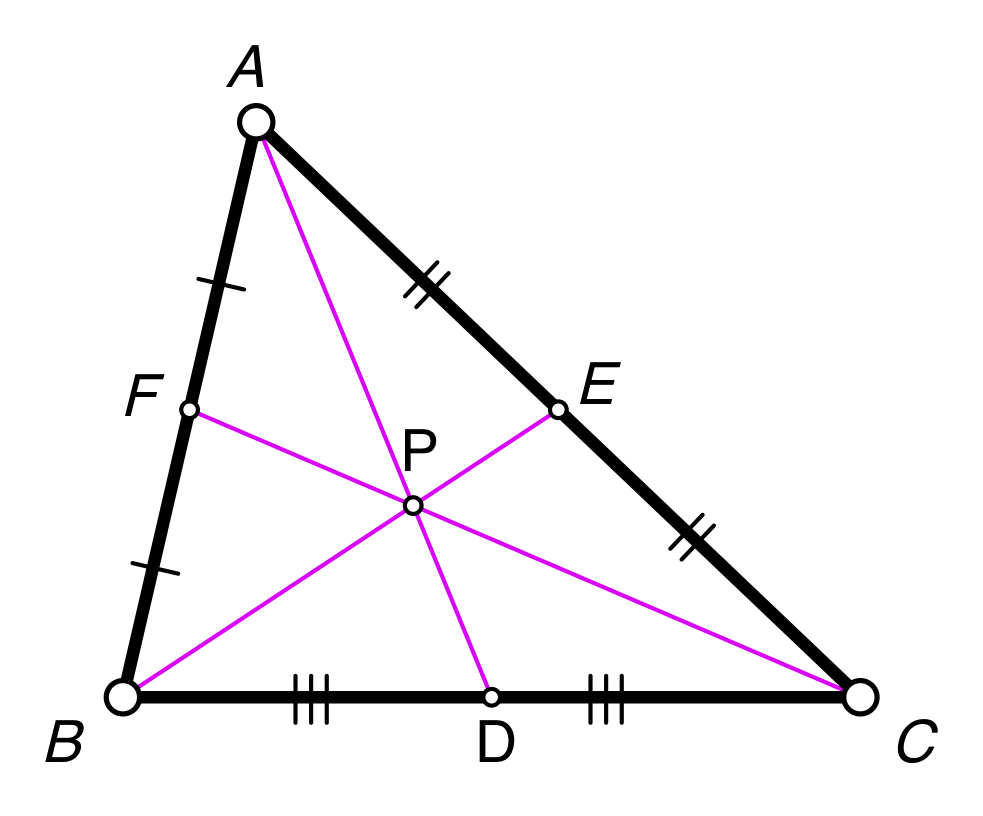}
\caption{six triangles have same area}
\label{fig:medians}
\end{figure}

\begin{proof}
If $XYZ$ is a triangle, then $[XYZ]$ will denote the area of that triangle.
If two triangles have the same altitude, then the ratio of their areas is the same as
the ratio of their bases. Thus $[PBD]=[PDC]$.
Since the centroid divides a median in the ratio $2:1$, this means that $[CEP]=\frac{1}{2}[CPB]=[CPD]$. Thus triangles 1, 2, and 3 have the same area. In the same manner, we see that all six triangles have the same area.
\end{proof} 

\begin{lemma}
\label{lemma:perims}
If $P$ is the centroid of $\triangle ABC$, then $s_1+s_3+s_5=s_2+s_4+s_6$.
\end{lemma}

\begin{proof}
We have the following six equations for the perimeters of the six triangles.

$$
\begin{aligned}
2s_1&=PB+PD+BD,\qquad&2s_2&=PD+PC+DC,\\
2s_3&=PC+PE+CE,&2s_4&=PE+PA+EA,\\
2s_5&=PA+PF+AF,&2s_6&=PF+PB+FB.
\end{aligned}
$$

Thus
$2s_1+2s_3+2s_5-(2s_2+2s_4+2s_6)=(BD-DC)+(CE-EA)+(AF-FB)=0$
and the lemma follows.\end{proof}

We can now state the relationship between the $r_i$ when $P$ is the centroid.
This theorem is attributed to Reidt in \cite[p.~618]{Dalle}.

\begin{theorem}
\label{thm:Centroid}
If $P$ is the centroid of $\triangle ABC$, then
$$\frac{1}{r_1}+\frac{1}{r_3}+\frac{1}{r_5}=\frac{1}{r_2}+\frac{1}{r_4}+\frac{1}{r_6}.$$
\end{theorem}

\begin{proof}
Recall that if $r$, $s$, and $K$ are the inradius, semiperimeter, and area of a triangle, respectively, then $r=K/s$.
From Lemma \ref{lemma:medians}, the six triangles have the same area. Call this area $K$.
From Lemma \ref{lemma:perims}, $s_1+s_3+s_5=s_2+s_4+s_6$.
Divide both sides of this equation by $K$ and use the fact that $1/r_i=s_i/K_i$ to get
the desired identity.
\end{proof}

We note a similar result from \cite[p.~618]{Dalle}.
\begin{theorem}
\label{thm:OddEvenReciprocals}
If $P$ is the centroid of $\triangle ABC$, then
$R_1R_3R_5=R_2R_4R_6$.
\end{theorem}

\bigskip 
\section{The Circumcenter}

Now we will consider the case when $P$ is the circumcenter.

\begin{lemma}
\label{lemma:circumtrig}
Let $P$ be the circumcenter of $\triangle ABC$ and let $R$ be its circumradius.
Let $\angle PAB=\angle PBA=\alpha$ and $\angle PBC=\beta$ (Figure \ref{fig:circumtrig}).
Then $R/r_1=\cot\alpha+\cot\frac{\beta}{2}.$
\end{lemma}

\bigskip
\begin{figure}[h!]
\centering
\includegraphics[width=0.4\linewidth]{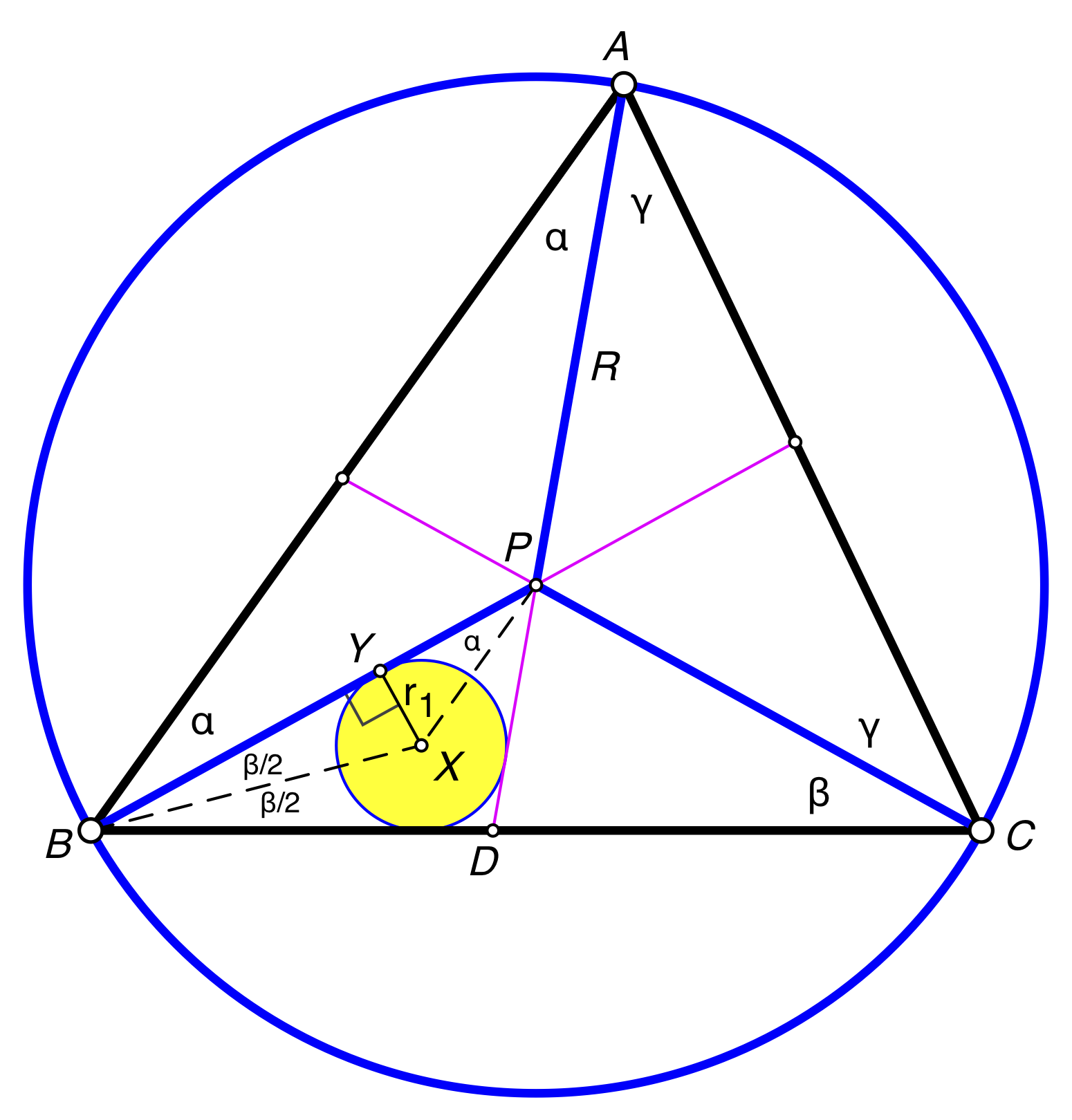}
\caption{}
\label{fig:circumtrig}
\end{figure}

\begin{proof}
Let $X$ be the center of incircle 1 (Figure \ref{fig:circumtrig}). Let $Y$ be the foot of the perpendicular from $X$ to $PB$.
Since $BX$ is the bisector of $\angle PBD$, $\angle PBX=\beta/2$.
Since $\angle BPD=\angle PAB+\angle PBA$, $\angle BPX=\alpha$.
Then $PY=r_1\cot\alpha$ and $YB=r_1\cot\frac{\beta}{2}$. Since $PB=R$, this gives
$R=r_1\cot\alpha+r_1\cot\frac{\beta}{2}$ and the result follows.
\end{proof}

We now give the relationship between the $r_i$ when $P$ is the circumcenter.

\begin{theorem}
\label{thm:Circumcenter}
If $P$ is the circumcenter of $\triangle ABC$ (Figure \ref{fig:circumcenter}), then
$$\frac{1}{r_1}+\frac{1}{r_3}+\frac{1}{r_5}=\frac{1}{r_2}+\frac{1}{r_4}+\frac{1}{r_6}.$$
\end{theorem}

\bigskip
\begin{figure}[h!]
\centering
\includegraphics[width=0.4\linewidth]{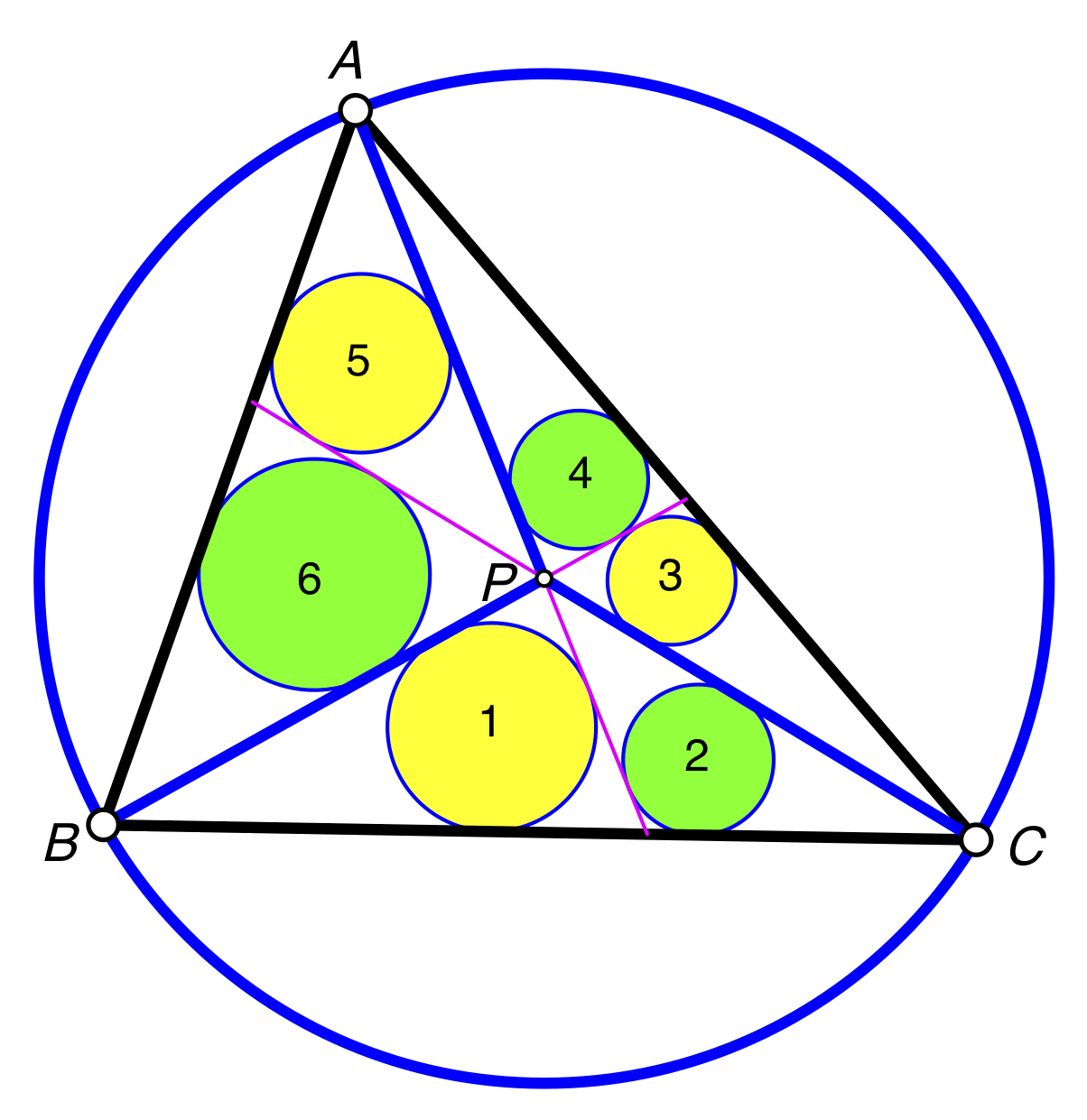}
\caption{$\frac{1}{r_1}+\frac{1}{r_3}+\frac{1}{r_5}=\frac{1}{r_2}+\frac{1}{r_4}+\frac{1}{r_6}$}
\label{fig:circumcenter}
\end{figure}

\begin{proof}
Let $R$ be the circumradius of $\triangle ABC$.
Let $\angle PAB=\angle PBA=\alpha$, $\angle PBC=\angle PCB=\beta$,
and $\angle PCA=\angle PAC=\gamma$ (Figure \ref{fig:circumtrig}).
By Lemma \ref{lemma:circumtrig}, we have the following six relationships.
$$
\begin{aligned}
R/r_1&=\cot\alpha+\cot\frac{\beta}{2},\qquad&
R/r_2&=\cot\gamma+\cot\frac{\beta}{2},\\
R/r_3&=\cot\beta+\cot\frac{\gamma}{2},&
R/r_4&=\cot\alpha+\cot\frac{\gamma}{2},\\
R/r_5&=\cot\gamma+\cot\frac{\alpha}{2},&
R/r_6&=\cot\beta+\cot\frac{\alpha}{2}.
\end{aligned}
$$
Adding the equations on the left gives the same result as adding the equations on the right. Dividing out the common factor $R$ proves the theorem.
\end{proof}

We mention the relationship between the $R_i$ when $P$ is the circumcenter.

\begin{theorem}
If $P$ is the circumcenter of $\triangle ABC$, then $R_1=R_2$. 
\end{theorem}

\begin{proof}
Let $R$ be the circumradius of $\triangle ABC$, so that $PA=PB=PC=R$.
Let $\angle BDP=\alpha$ and $\angle PDC=\beta$ (Figure \ref{fig:twoEqualCircles}). By the Extended Law of Sines in $\triangle PBD$, $R/\sin\alpha=2R_1$. Similarly for $\triangle PCD$, $R/\sin\beta=2R_2$. But since $\alpha$ and $\beta$ are supplementary, $\sin\alpha=\sin\beta$. Therefore, $R_1=R_2$.
\end{proof}

\bigskip
\begin{figure}[h!]
\centering
\includegraphics[width=0.5\linewidth]{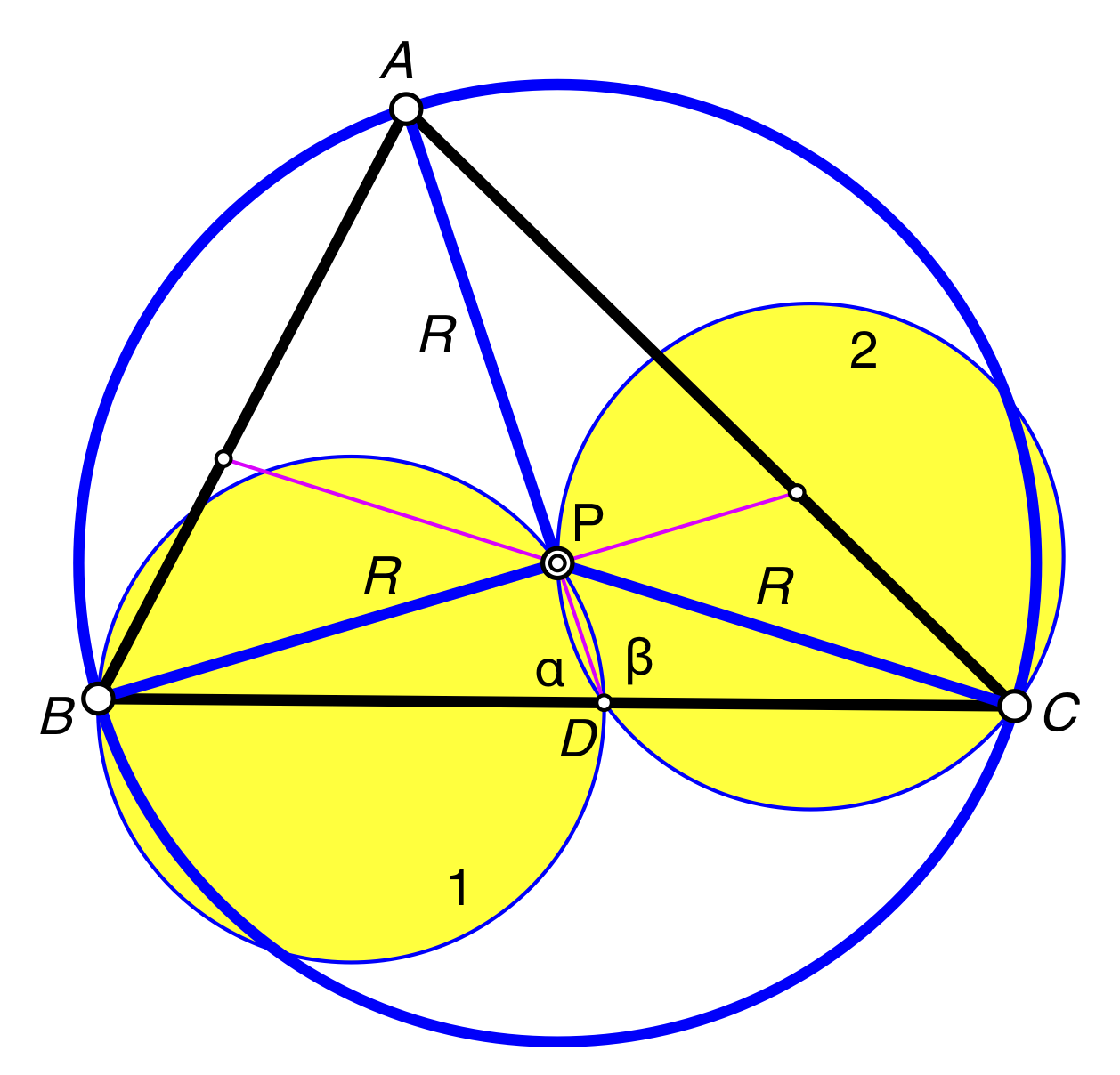}
\caption{Yellow circles are congruent}
\label{fig:twoEqualCircles}
\end{figure}

\bigskip 
\section{The Incenter}

Next, we will consider some cases when $P$ is the incenter.

\begin{theorem}
\label{thm:incenter60}
If $P$ is the incenter of $\triangle ABC$ and $\angle ABC=60\degrees$ (Figure \ref{fig:incenter60}), then
$$\frac{1}{r_1}+\frac{1}{r_4}+\frac{1}{r_5}=\frac{1}{r_2}+\frac{1}{r_3}+\frac{1}{r_6}.$$
\end{theorem}

\bigskip
\begin{figure}[h!]
\centering
\includegraphics[width=0.7\linewidth]{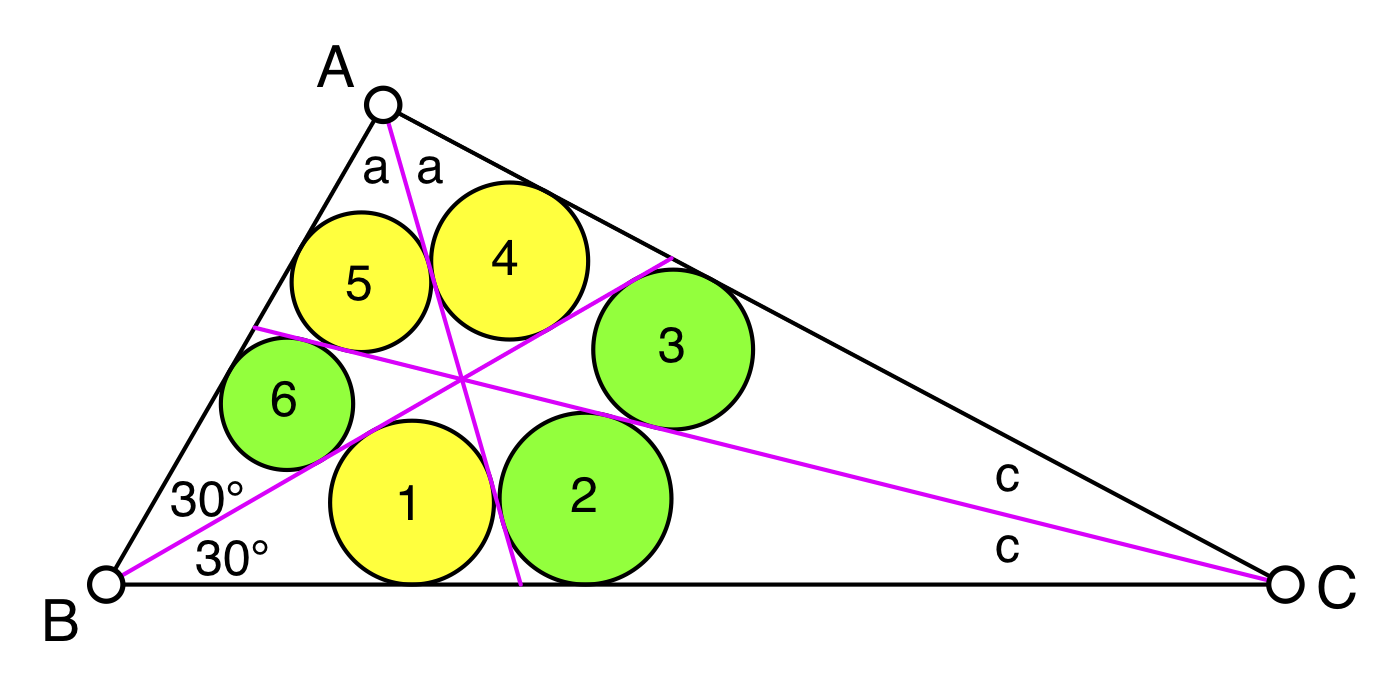}
\caption{$\frac{1}{r_1}+\frac{1}{r_4}+\frac{1}{r_5}=\frac{1}{r_2}+\frac{1}{r_3}+\frac{1}{r_6}$}
\label{fig:incenter60}
\end{figure}

\begin{proof}
Let $\angle BAD=\angle DAC=a$, $\angle ABE=\angle EBC=b$, and $\angle ACF=\angle FCB=c$,
as shown in Figure \ref{fig:incenterProof}. Note that $a+b+c=\pi/2$.
By the Extended Law of Sines in $\triangle ABC$, $AB/\sin 2c=2R$, where $R$ is the radius of the circumcircle of $\triangle ABC$.
Since the result we want to prove is invariant under scaling, without loss of generality, we may assume that $R=1/2$. Thus $AB=\sin 2c$. In the same manner, we can find $AC$ and $BC$.
We get the following.

$$
AB=\sin 2c,\quad BC=\sin 2a, \quad CA=\sin 2b.
$$

\bigskip
\begin{figure}[h!]
\centering
\includegraphics[width=0.5\linewidth]{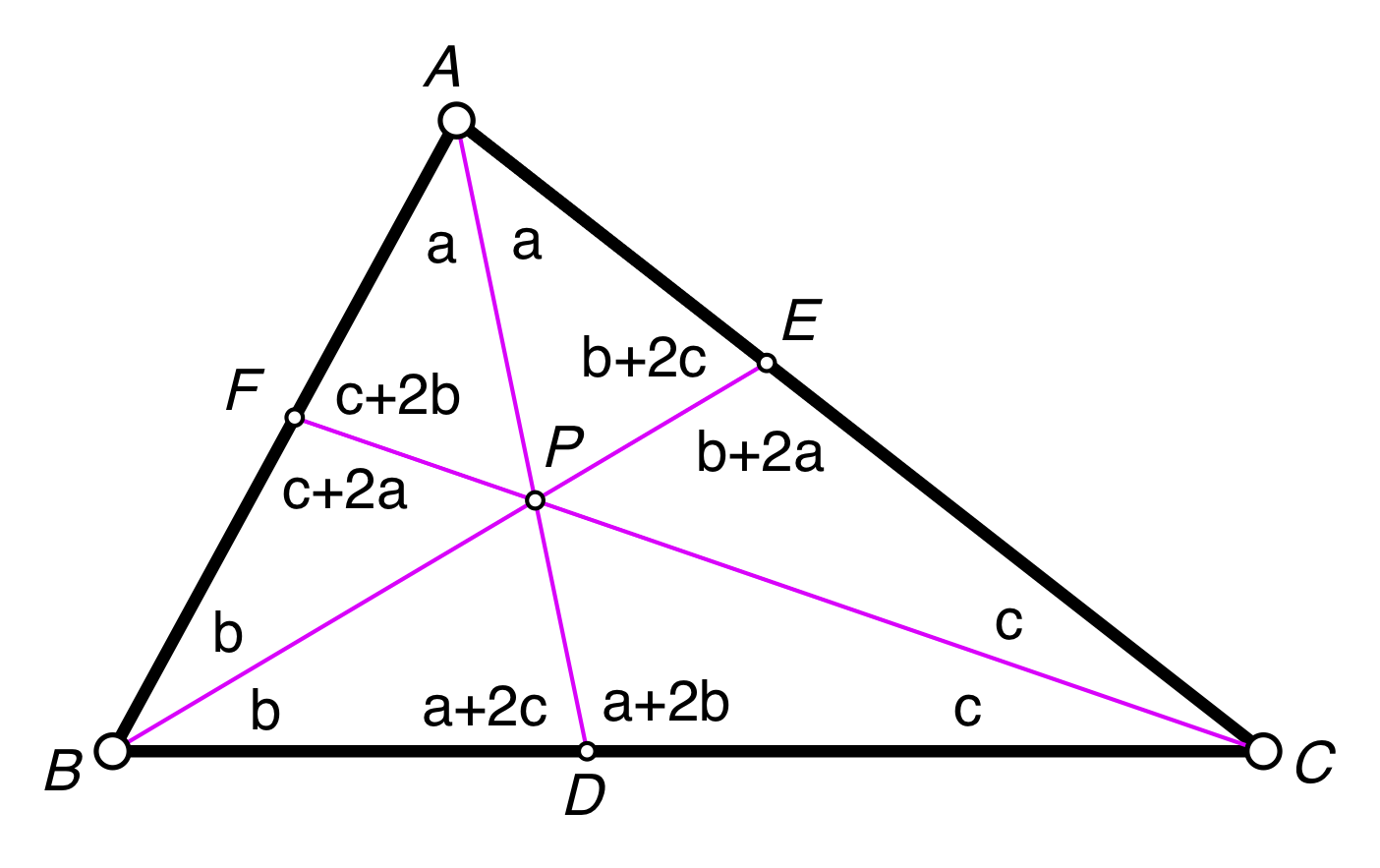}
\caption{}
\label{fig:incenterProof}
\end{figure}

Applying the Law of Sines to $\triangle ABD$ gives $BD/\sin a=AB/\sin(a+2c)$. This allows us to compute $BD$. In a similar manner, we get the following.

$$
\begin{aligned}
BD=\frac{\sin a\sin 2c}{\sin(a+2c)},\quad&CE=\frac{\sin b\sin 2a}{\sin(b+2a)},&AF=\frac{\sin c\sin 2b}{\sin(c+2b)},\\
CD=\frac{\sin a\sin 2b}{\sin(a+2b)},\quad&AE=\frac{\sin b\sin 2c}{\sin(b+2c)},&BF=\frac{\sin c\sin 2a}{\sin(c+2a)}.
\end{aligned}
$$

\bigskip
Note that $\angle BPD=a+b$. Applying the Law of Sines to $\triangle BPD$ allows us to compute the values of $PB$ and $PD$. In the same way, we can compute $PC$, $PE$, $PA$, and $PF$.
We get the following.

$$
\begin{aligned}
PA=\frac{\sin c\sin 2b}{\sin(c+a)},\quad&PD=\frac{\sin a\sin b\sin 2c}{\sin(a+b)\sin(a+2c)},\\
PB=\frac{\sin a\sin 2c}{\sin(a+b)},\quad&PE=\frac{\sin b\sin c\sin 2a}{\sin(b+c)\sin(b+2a)},\\
PC=\frac{\sin b\sin 2a}{\sin(b+c)},\quad&PF=\frac{\sin c\sin a\sin 2b}{\sin(c+a)\sin(c+2b)}.
\end{aligned}
$$

\medskip
We now have expressions for the length of every line segment in the figure in terms of $a$, $b$, and $c$. Thus, the perimeters of all the triangles are known and we have expressed each of the $s_i$
in terms of $a$, $b$, and $c$. The areas of the triangles can also be found. For example, $K_1=\frac{1}{2}PB\cdot BD\sin b$. Knowing all the $s_i$ and $K_i$ lets us find the values of all
the $r_i$, since $r_i=K_i/s_i$.

\medskip
We can plug these values for the $r_i$ into the expression
$$\frac{1}{r_1}+\frac{1}{r_4}+\frac{1}{r_5}-\left(\frac{1}{r_2}+\frac{1}{r_3}+\frac{1}{r_6}\right).$$
Letting $a=\pi/2-b-c$ and $b=\pi/6$ then gives us an expression with $c$ as the only variable.
Simplifying this expression (using a symbolic algebra system), we find that the result is 0, thus proving our theorem.
\end{proof}

\begin{theorem}
\label{thm:incenter120}
If $P$ is the incenter of $\triangle ABC$ and $\angle ABC=120\degrees$ (Figure \ref{fig:incenter120}), then
$$r_1r_2r_3+r_3r_4r_5+r_3r_4r_6=r_1r_3r_4+r_2r_3r_4+r_4r_5r_6.$$
\end{theorem}

This can also be written as
$\displaystyle r_1+r_2+\frac{r_5r_6}{r_3}=r_5+r_6+\frac{r_1r_2}{r_4}$.

\bigskip
\begin{figure}[h!]
\centering
\includegraphics[width=0.7\linewidth]{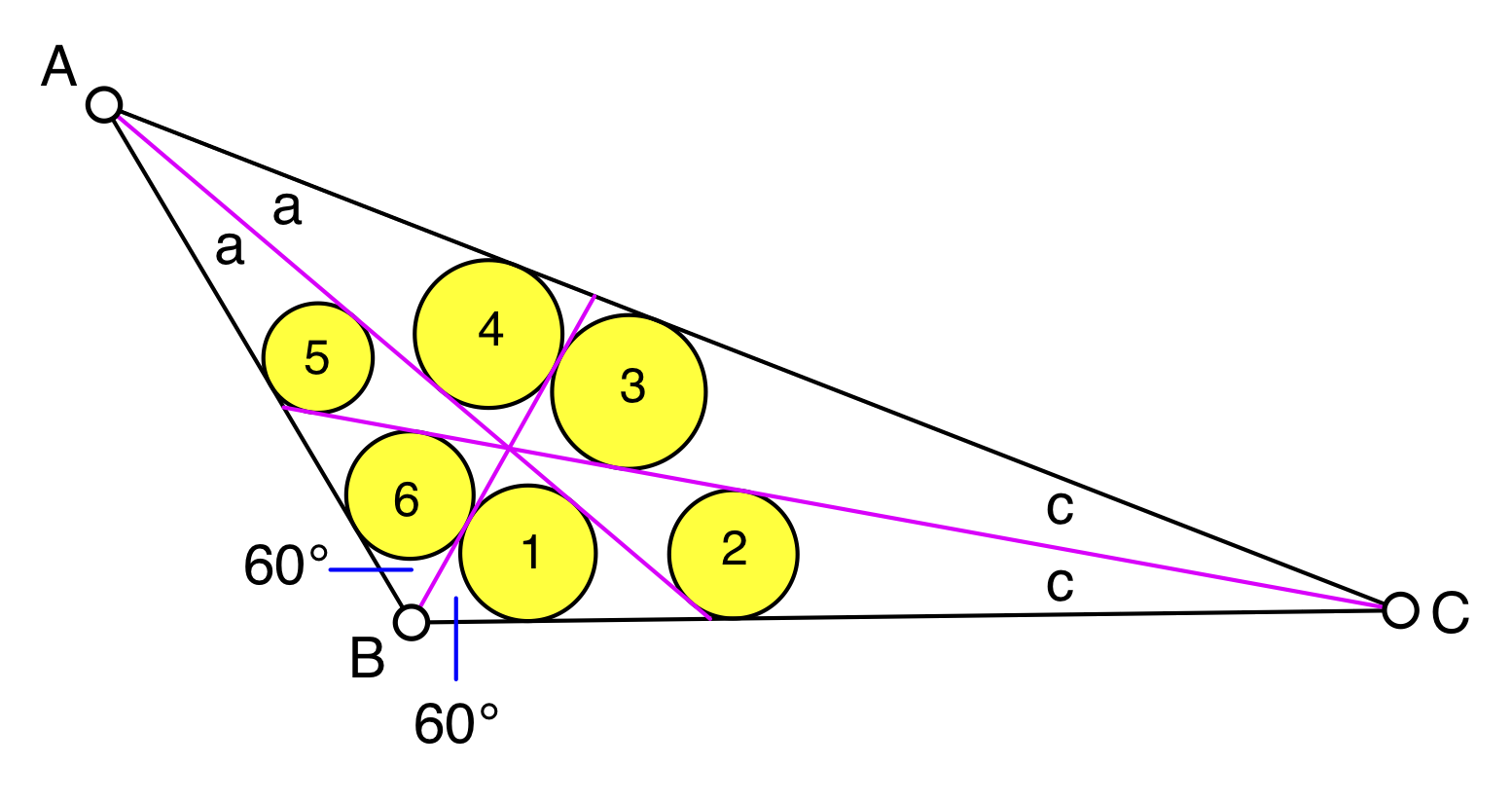}
\caption{$r_1r_2r_3+r_3r_4r_5+r_3r_4r_6=r_1r_3r_4+r_2r_3r_4+r_4r_5r_6$}
\label{fig:incenter120}
\end{figure}

\begin{proof}
This theorem can be proven using the same procedure that was used to prove Theorem \ref{thm:incenter60}.
The details are omitted.
\end{proof}

\begin{open}
Is there a simple relationship between the $r_i$ that holds for all triangles when $P$
is the incenter?
\end{open}

Note that there are two independent variables, $a$ and $b$, and six equations representing the values of the $r_i$. Thus, variables $a$ and $b$ can be eliminated resulting in an equation relating the $r_i$. Since there are so many more equations than variables, multiple relationships can be found. For example, in the $30\degrees$--$60\degrees$--$90\degrees$ right triangle, we have a number of simple relationships between the $r_i$, as shown by the following theorem.

\begin{theorem}
\label{thm:incenter306090}
If $P$ is the incenter of $\triangle ABC$ and $\angle ABC=30\degrees$ and $\angle ACB=60\degrees$, then the $r_i$ are related to each other by each of the following equations.
$$
\begin{aligned}
\frac{2}{r_1^4}+\frac{8}{r_3^4}+\frac{6}{r_4^4}+\frac{5}{r_5^4}&=\frac{14}{r_2^4}+\frac{20}{r_6^4},\\
\frac{3}{r_1^2}+\frac{3}{r_2^2}+\frac{3}{r_4^2}&=\frac{3}{r_3^2}+\frac{2}{r_5^2}+\frac{3}{r_6^2},\\
\frac{2}{r_1}+\frac{2}{r_2}+\frac{1}{r_5}+\frac{1}{r_6}&=\frac{3}{r_3}+\frac{2}{r_4},\\
3r_1+5r_3+55r_4+22r_6&=4r_2+75r_5,\\
2r_1r_3+3r_2r_4+9r_5r_1+9r_6r_2&=27r_3r_5+r_4r_6.
\end{aligned}
$$
\end{theorem}

\begin{proof}
For this triangle, the values of the $r_i$ are as follows.
$$
\begin{aligned}
r_1&=\frac{1}{4} \left(\sqrt2-1\right) \left(\sqrt3-1\right),\\
r_2&=\frac{1}{4}\left(-10-7\sqrt2+6\sqrt3+4\sqrt6\right),\\
r_3&=\frac{1}{4}\left(9-7\sqrt2-5\sqrt3+4\sqrt6\right),\\
r_4&=\frac{1}{8}\left(-4-\sqrt2+2\sqrt3+\sqrt6\right),\\
r_5&=\frac{1}{24}\left(-3\sqrt2+2\sqrt3+\sqrt6\right),\\
r_6&=\frac{1}{4}\left(1+\sqrt3-\sqrt6\right).\\
\end{aligned}
$$
These values can be substituted into the stated equations to verify the results
(using computer simplification, as necessary).
\end{proof}

Here is an approach that might be used to find the general relationship between the $r_i$
when $P$ is the incenter. To avoid fractions, we will replace $a$, $b$, and $c$ from Figure \ref{fig:incenterProof} by $2a$, $2b$, and $2c$ to get Figure \ref{fig:trig}. We can express $PB$ as the sum of two lengths using circles 6 and 1 in two different ways. Equating these expressions gives the following equation.

\begin{equation}
\label{eq:trig}
\frac{r_6}{r_1}=\frac{\cot b+\cot(a+b)}{\cot b+\cot(b+c)}.
\end{equation}

\bigskip
\begin{figure}[h!]
\centering
\includegraphics[width=0.9\linewidth]{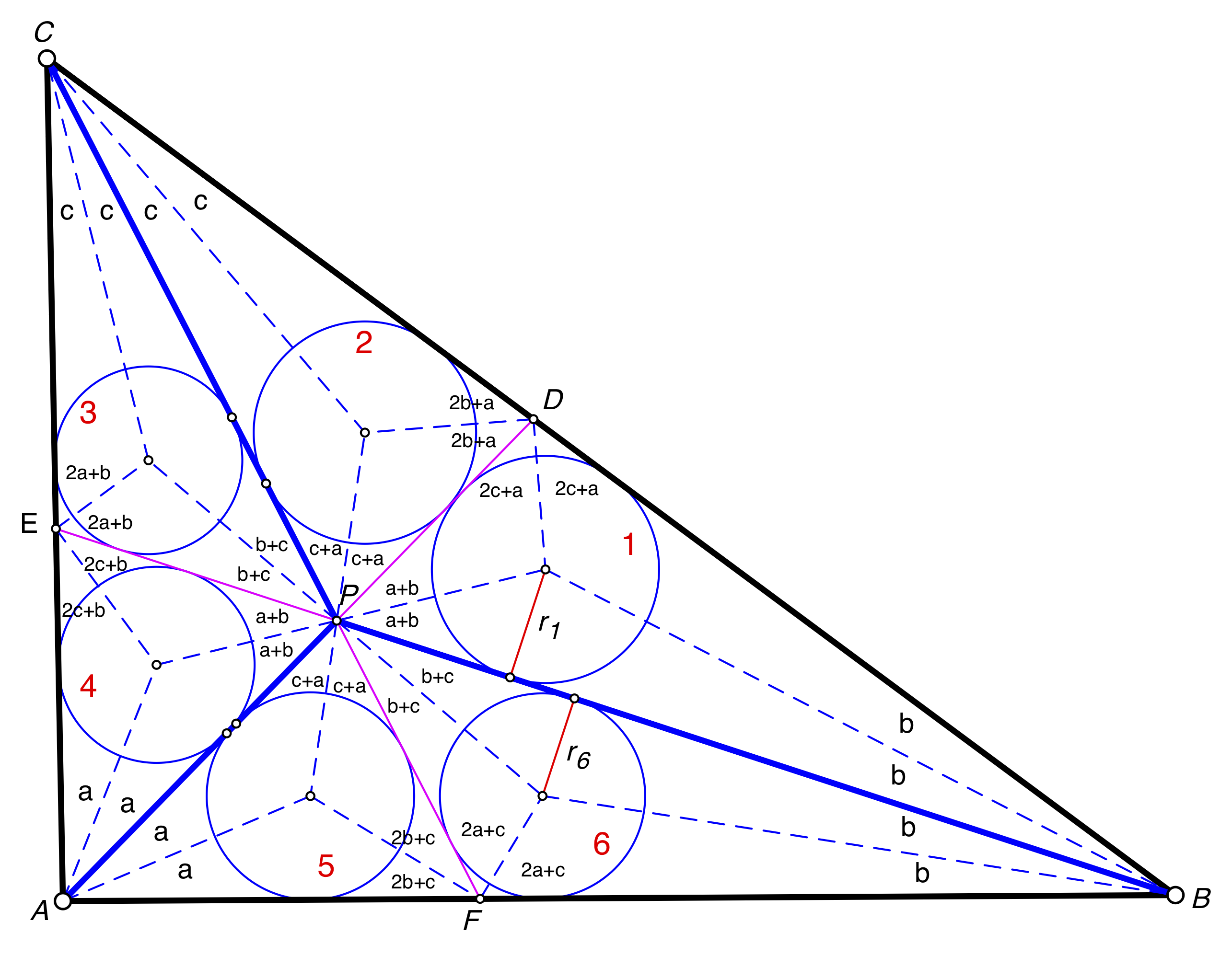}
\caption{}
\label{fig:trig}
\end{figure}

Since $4a+4b+4c=180\degrees$, $c=45\degrees-a-b$. Substitute this value of $c$
into equation~(\ref{eq:trig}). Then use the addition formula for cotangent,
$$\cot(x+y)=\frac{\cot x\cot y-1}{\cot x+\cot y},$$
to write all trigonometric expressions in terms of $\cot a$ and $\cot b$.

\medskip
In a similar manner we can form two other equations for $r_2/r_3$ and $r_4/r_5$.
This gives us the following three equations for $u=r_6/r_1$, $v=r_2/r_3$, and $w=r_4/r_5$
in terms of the two unknowns $C_a=\cot a$ and $C_b=\cot b$.

$$
\begin{aligned}
u&=\frac{(C_a-1)(C_b^2+2C_aC_b-1)}{(C_a+C_b)(1+C_a-C_b+C_aC_b)},\\
v&=\frac{(C_b-1)(C_bC_a^2-2C_a-C_b)}{(C_a-1)(C_aC_b^2-2C_b-C_a)},\\
w&=\frac{(C_a+C_b)(1-C_a+C_b+C_aC_b)}{(C_b-1)(C_a^2+2C_aC_b-1)}.
\end{aligned}
$$

Clearing fractions gives us three polynomial equations in the variables $C_a$ and $C_b$.
In theory, we should be able to eliminate $C_a$ and $C_b$ from these three equations,
leaving us with a single equation relating $u$, $v$, and $w$. This equation would
be the desired relationship between the $r_i$. I have not been able to perform
this elimination.

\bigskip
The following may be a simpler question.
 
\begin{open}
Find the relationship between $r_1$, $r_2$, and $r_3$ that holds for all triangles when $P$
is the incenter.
\end{open}

Such a relationship should exist because we can express each of $r_1$, $r_2$, and $r_3$
in terms of $a$ and $b$. This would give us three equations in two unknowns. In theory,
we should be able to eliminate $a$ and $b$ from these three equations, giving us a single
equation relating $r_1$, $r_2$, and $r_3$.

\bigskip 
\section{Other Points}

Next, we will consider other points inside triangle $ABC$.
We start with an example where there is a linear relationship between the $r_i$.

\begin{theorem}
\label{thm:linear}
If $P$ is a point inside $\triangle ABC$, and $\angle ABP=10\degrees$, $\angle PBC=30\degrees$, $\angle BCP=80\degrees$, and $\angle PCA=20\degrees$ (Figure \ref{fig:linear}), then
$$5r_1+6r_2+r_4=r_3+3r_5+15r_6.$$
\end{theorem}

\bigskip
\begin{figure}[h!]
\centering
\includegraphics[width=0.6\linewidth]{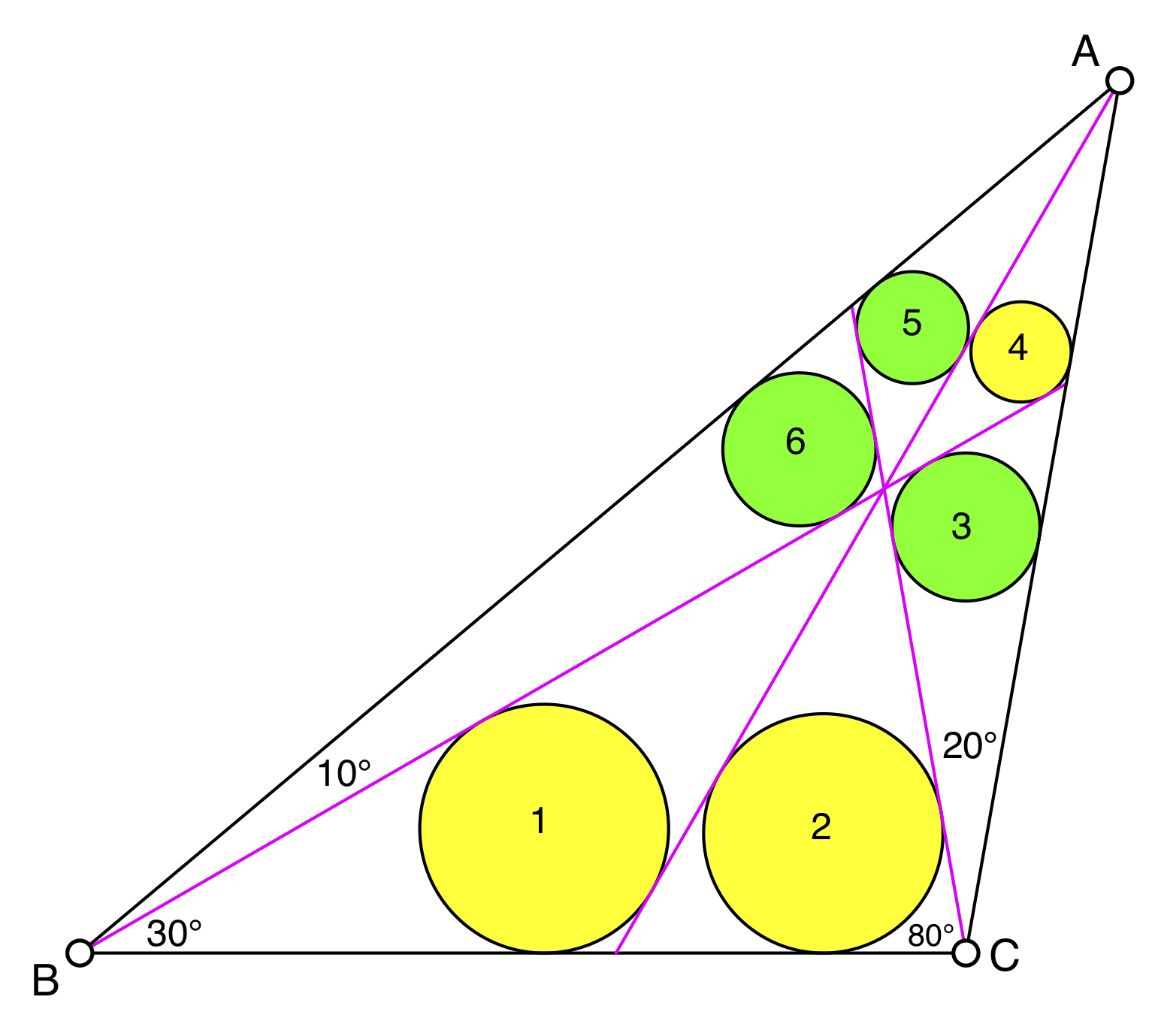}
\caption{$5r_1+6r_2+r_4=r_3+3r_5+15r_6$}
\label{fig:linear}
\end{figure}

\begin{proof}
We follow the same general procedure that was used in the proof of Theorem \ref{thm:incenter60}.
Let $\angle ABE=a$, $\angle EBC=b$, $\angle BCF=c$, and $\angle FCA=d$,
as shown in Figure \ref{fig:genProof}. Note that
$\angle BAC=\pi-a-b-c-d$, $\angle AFC=a+b+c$, $\angle AEB=b+c+d$, and $\angle EPC=b+c$.

\medskip
Without loss of generality, we may assume that $R=1/2$, where $R$ is the radius of the circumcircle of $\triangle ABC$.
By the Extended Law of Sines we get the following.

$$
AB=\sin(c+d),\quad AC=\sin(a+b),\quad BC=\sin(a+b+c+d).
$$

\bigskip
\begin{figure}[h!]
\centering
\includegraphics[width=0.5\linewidth]{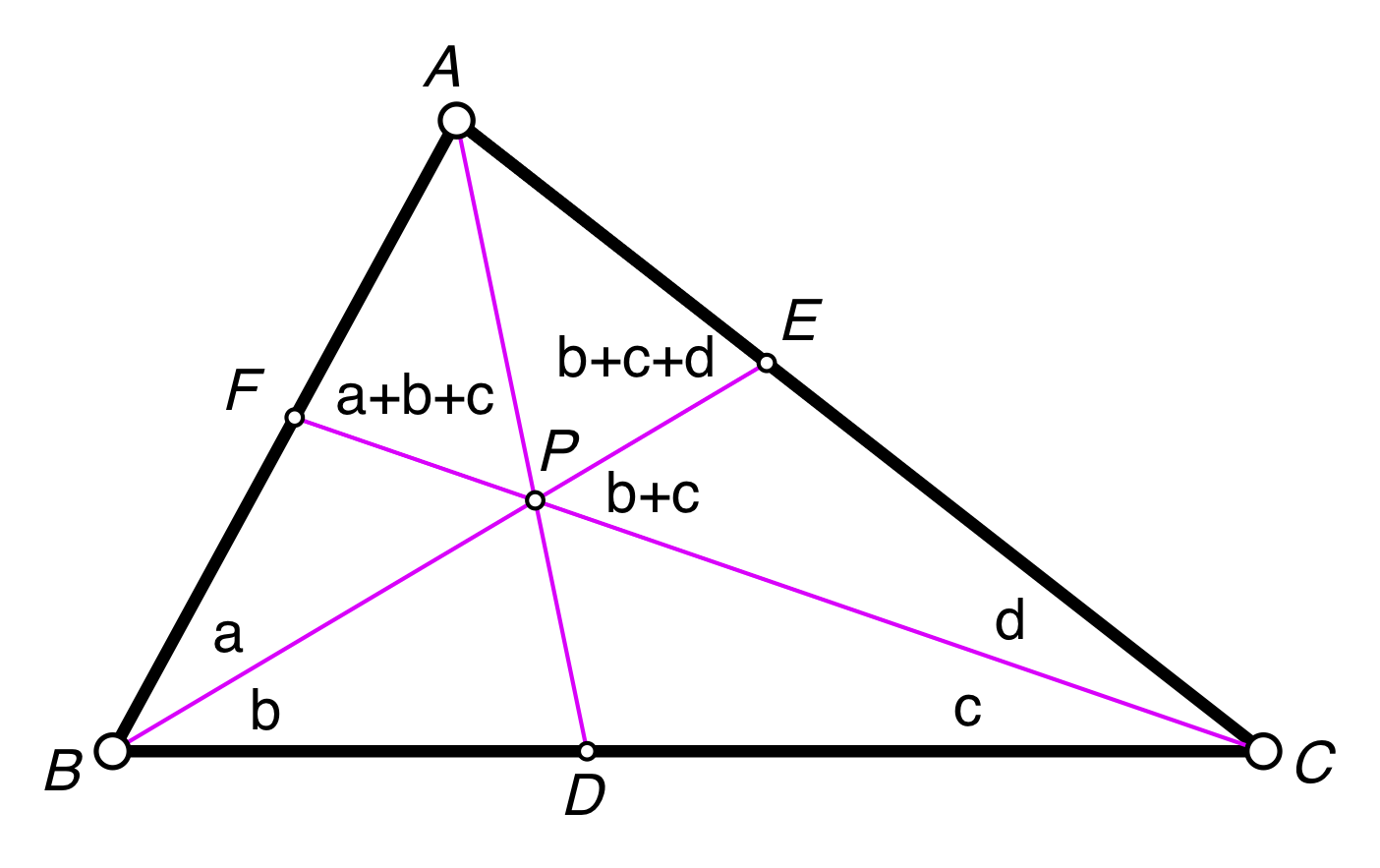}
\caption{}
\label{fig:genProof}
\end{figure}

\goodbreak
We then use the Law of Sines to get the following.

$$
\begin{aligned}
AF=\frac{\sin(a+b)\sin d}{\sin(a+b+c)},\quad&BF=\frac{\sin(a+b+c+d)\sin c}{\sin(a+b+c)},\\
AE=\frac{\sin(c+d)\sin a}{\sin(b+c+d)},\quad&CE=\frac{\sin(a+b+c+d)\sin b}{\sin(b+c+d)}.
\end{aligned}
$$

Applying the Law of Sines again gives the following.

$$
\begin{aligned}
PB=\frac{\sin(a+b+c+d)\sin c}{\sin (b+c)},\quad&PE=\frac{CE\sin d}{\sin(b+c)},\\
PC=\frac{\sin(a+b+c+d)\sin b}{\sin(b+c)},\quad&PF=\frac{BF\sin a}{\sin(b+c)}.
\end{aligned}
$$

By Ceva's Theorem, $\frac{BD}{CD}=\frac{BF\cdot AE}{AF\cdot CE}$.
This gives the following.

$$
\begin{aligned}
BD&=\frac{BF\cdot AE\cdot BC}{BF\cdot AE+AF\cdot CE},\\
CD&=\frac{AF\cdot CE\cdot BC}{BF\cdot AE+AF\cdot CE}.
\end{aligned}
$$

Length $PA$ is calculated using the Law of Cosines in triangle $APB$.
We get the following.

$$
PA=\sqrt{AB^2+PB^2-2\cdot AB\cdot PB\cdot \cos a}.
$$

To find the length of $PD$ without introducing  another square root, we can apply Menelaus' Theorem to
transversal $BPE$ in $\triangle ADC$. We get the following.

$$
PD=\frac{PA\cdot BD\cdot CE}{BC\cdot AE}.
$$

We now have expressions for the length of every line segment in the figure in terms of $a$, $b$, $c$, and $d$. We can therefore calculate all the $s_i$, $K_i$, and $r_i$.

\medskip
We can plug these values for the $r_i$ into the expression
$$5r_1+6r_2+r_4-\left(r_3+3r_5+15r_6\right).$$
Letting $a=10\degrees$, $b=30\degrees$, $c=80\degrees$, and $d=20\degrees$ gives us an expression with no variables.
Simplifying this expression (using a symbolic algebra system), we find that the result is 0, thus proving our theorem.
\end{proof}

Sometimes the relationship between the $r_i$ is more regular as in the following two theorems. The proofs are similar to the proof of Theorem \ref{thm:linear}.
When the formulas for the lengths of the radii previously found are applied to the equation to be proved, the result is a trigonometric equation that can be proven to be an identity using symbolic algebra computation. The details are omitted.

\begin{theorem}
\label{thm:6thinvIncircles}
If $P$ is a point inside $\triangle ABC$, and $\angle PBA$, $\angle PBC$, $\angle PCB$, and $\angle PCA$ are as shown in Figure \ref{fig:6invIncirclesD}, then
$$\frac{1}{r_1}+\frac{1}{r_3}+\frac{1}{r_4}=\frac{1}{r_2}+\frac{1}{r_5}+\frac{1}{r_6}.$$
\end{theorem}

\bigskip
\begin{figure}[h!]
\centering
\includegraphics[width=0.5\linewidth]{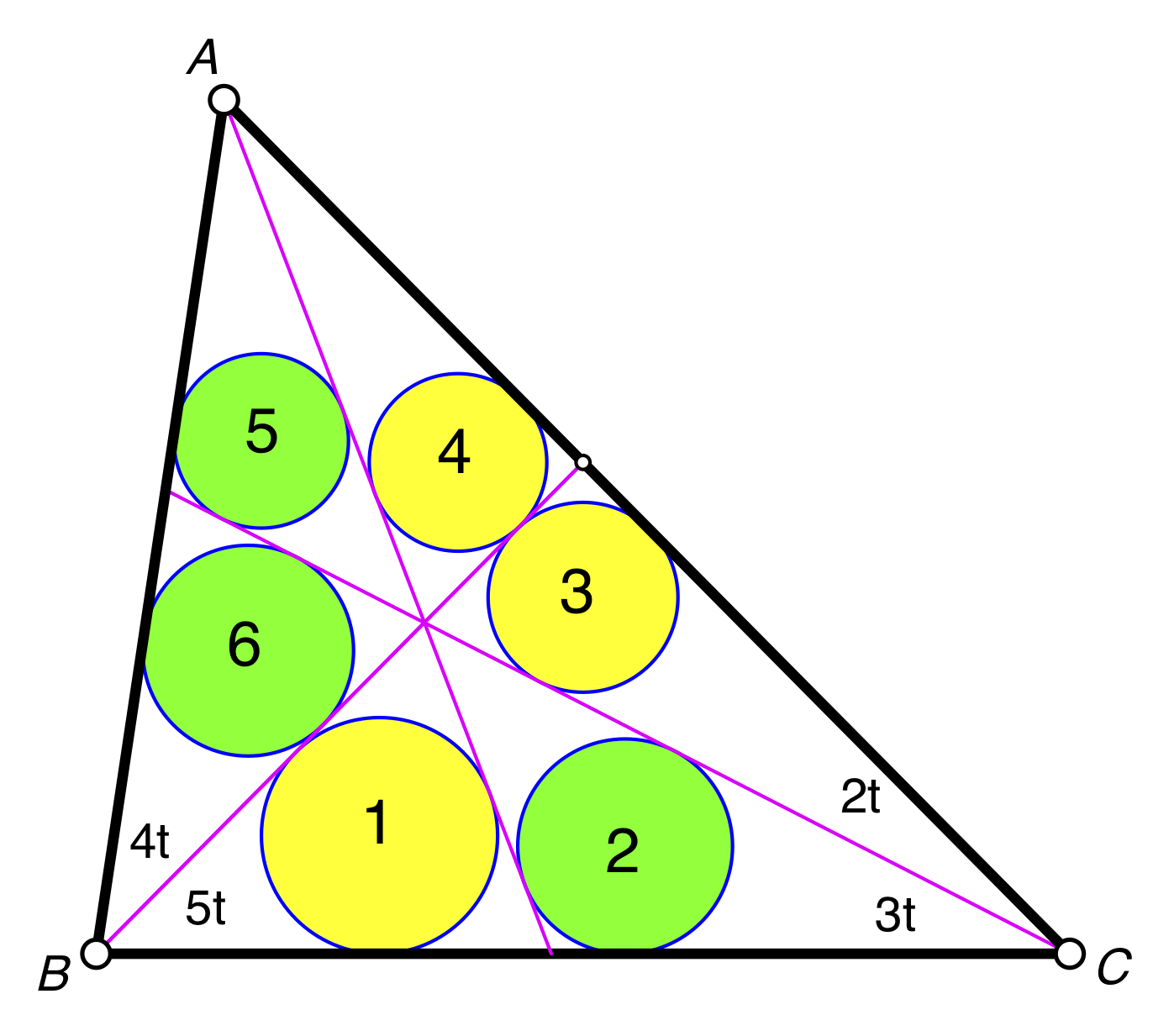}
\caption{$\frac{1}{r_1}+\frac{1}{r_3}+\frac{1}{r_4}=\frac{1}{r_2}+\frac{1}{r_5}+\frac{1}{r_6}$}
\label{fig:6invIncirclesD}
\end{figure}

\begin{theorem}
\label{thm:5invIncircles}
If $P$ is a point inside $\triangle ABC$, and $\angle PBA$, $\angle PBC$, $\angle PCB$, and $\angle PCA$ are as shown in any of the triangles depicted in Figure \ref{fig:parametric1}, then
$$\frac{1}{r_1}+\frac{1}{r_4}+\frac{1}{r_6}=\frac{1}{r_2}+\frac{1}{r_3}+\frac{1}{r_5}.$$
Note that in each of these examples, $t$ is an arbitrary parameter.
\end{theorem}

\bigskip
\begin{figure}[h!]
\centering
\includegraphics[width=0.9\linewidth]{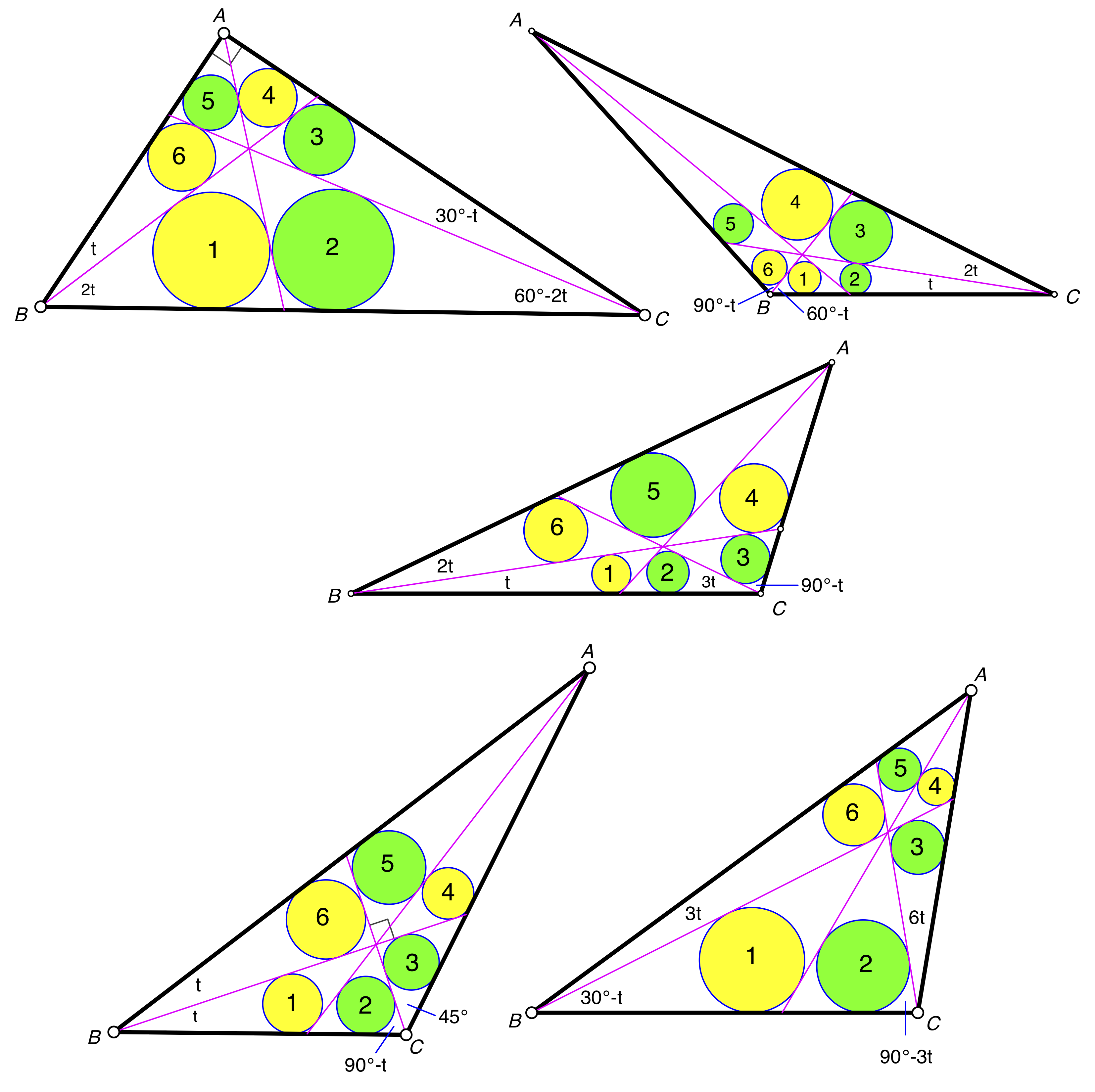}
\caption{$\frac{1}{r_1}+\frac{1}{r_4}+\frac{1}{r_6}=\frac{1}{r_2}+\frac{1}{r_3}+\frac{1}{r_5}$}
\label{fig:parametric1}
\end{figure}

The reader may be wondering how I came up with these results. Here is the procedure I used.

Let $f(r)$ denote some function of $r$ such as $r$, $r^2$, or $1/r$.
I varied all combinations of $a$, $b$, $c$, and $d$ (see Figure \ref{fig:genProof}) through all multiples of $1\degrees$ and calculated
$f(r_1)$ through $f(r_6)$. Using the Mathematica\textsuperscript{\textregistered} function \texttt{FindIntegerNullVector}, I looked for any linear relationships between these
six numbers. If a linear relationship existed which involved all six numbers,
I logged the quadruple $\langle a,b,c,d\rangle$ along with the coefficients of the relationship into a database. After all quadruples were examined, I looked at all pairs of entries in the database that had the same set of coefficients. If $Q_1=\langle a_1,b_1,c_1,d_1\rangle$ and $Q_2=\langle a_2,b_2,c_2,d_2\rangle$ were two such quadruples,
I then formed the quadruple $Q_3$ such that $\langle Q_1,Q_2,Q_3\rangle$ formed an arithmetic progression in each component. Then I examined $Q_3$ to see if it also satisfied the same linear combination. If it did, then this suggested a one-parameter family of solutions (which had
to be confirmed).

\medskip
Note that many solutions were found for various functions $f$, but no one-parameter families of solutions were found except when $f(r)=1/r$. It is not clear why this should be the case.

\medskip
We note several related results that hold whenever $P$ is an arbitrary point inside $\triangle ABC$.

\begin{theorem}
If $P$ is any point inside $\triangle ABC$ (Figure \ref{fig:oddEvenAngles}), then $K_1K_3K_5=K_2K_4K_6$.
\end{theorem}

\bigskip
\begin{figure}[h!]
\centering
\includegraphics[width=0.4\linewidth]{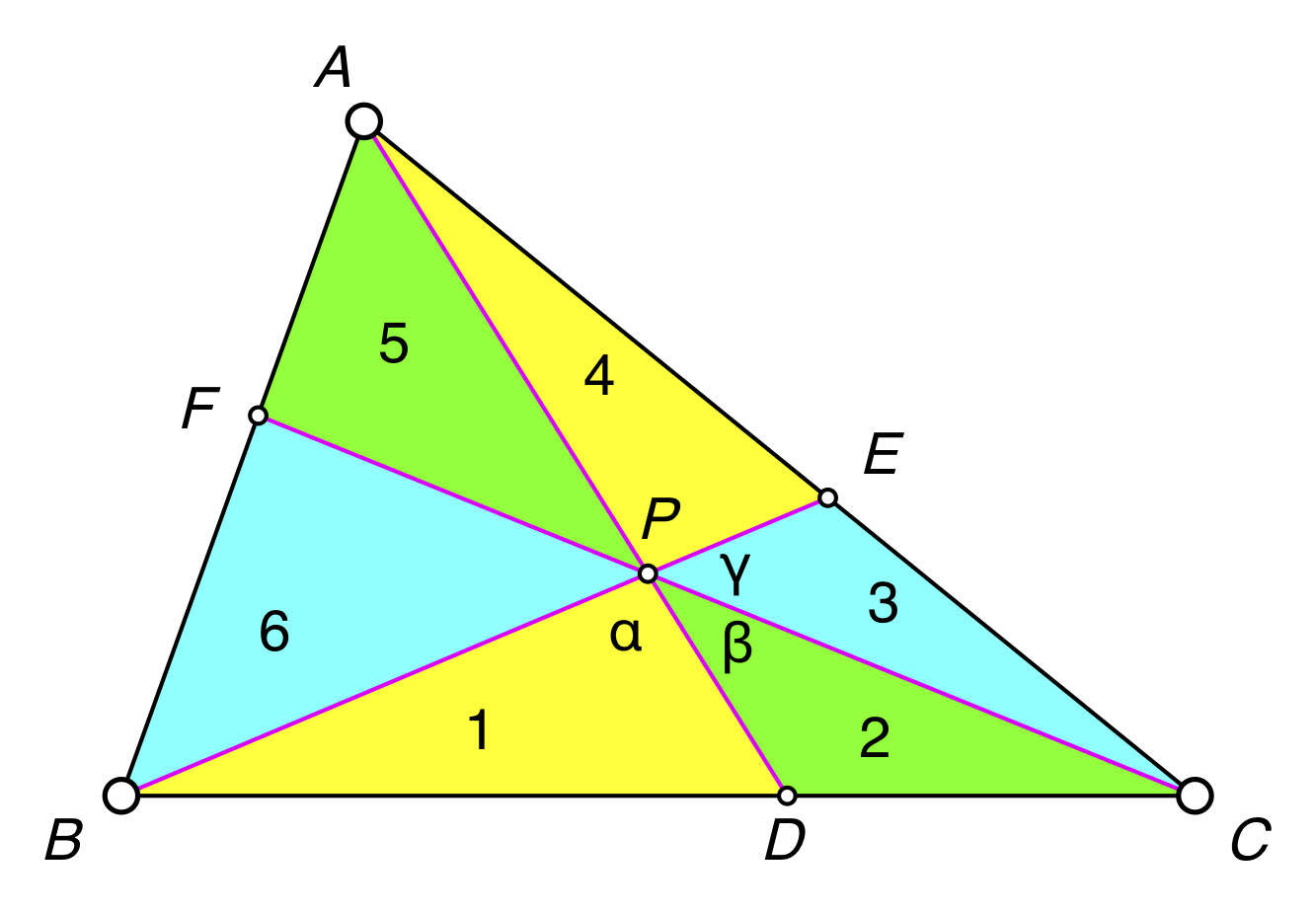}
\caption{$K_1K_3K_5=K_2K_4K_6$}
\label{fig:oddEvenAngles}
\end{figure}

\begin{proof}
Let $\angle BPD=\alpha$, $\angle DPC=\beta$, and $\angle CPE=\gamma$.
Noting that vertical angles are equal and
using the formula for the area of a triangle in terms of two sides and the sine of the included angle, we have the following six equations.
$$
\begin{aligned}
2K_1&=PB\cdot PD\cdot \sin\alpha,\qquad&
2K_2&=PD\cdot PC\cdot \sin\beta,\\
2K_3&=PC\cdot PE\cdot \sin\gamma,&
2K_4&=PE\cdot PA\cdot \sin\alpha,\\
2K_5&=PA\cdot PF\cdot \sin\beta,&
2K_6&=PF\cdot PB\cdot \sin\gamma.
\end{aligned}
$$
Multiplying gives
$$8K_1K_3K_5=(PB\cdot PD\cdot \sin\alpha)(PC\cdot PE\cdot \sin\gamma)(PA\cdot PF\cdot \sin\beta)$$
and
$$8K_2K_4K_6=(PD\cdot PC\cdot \sin\beta)(PE\cdot PA\cdot \sin\alpha)(PF\cdot PB\cdot \sin\gamma).$$
The right sides of both equations are equal, so $K_1K_3K_5=K_2K_4K_6$.
\end{proof} 

\begin{theorem}
\label{thm:relatedCentroid}
\cite[p.~43]{Mannheim} If $P$ is any point inside $\triangle ABC$, then
$$\frac{1}{K_1}+\frac{1}{K_3}+\frac{1}{K_5}=\frac{1}{K_2}+\frac{1}{K_4}+\frac{1}{K_6}.$$
\end{theorem}

\begin{open}
Is there a simple relationship between the $r_i$ that holds for all triangles and all points $P$ inside the triangle?
\end{open}

A simpler question may be the following.

\begin{open}
Is there a simple relationship between the $r_i$ that holds for all points $P$ inside an equilateral triangle?
\end{open}

\bigskip 
\section{Opportunities for Future Research}

When the three cevians through a point inside a triangle are extended to the cirumcircle,
other circles can be formed.
Figure \ref{fig:mixtilinear} shows some examples.
Each yellow circle associated with the triangle on the left is the incircle of a region bounded by two cevians and the circumcircle.
Each green circle associated with the triangle on the right is the incircle of a region bounded by one side of the triangle, one cevian, and the circumcircle.

\bigskip
\begin{figure}[h!]
\centering
\includegraphics[width=0.7\linewidth]{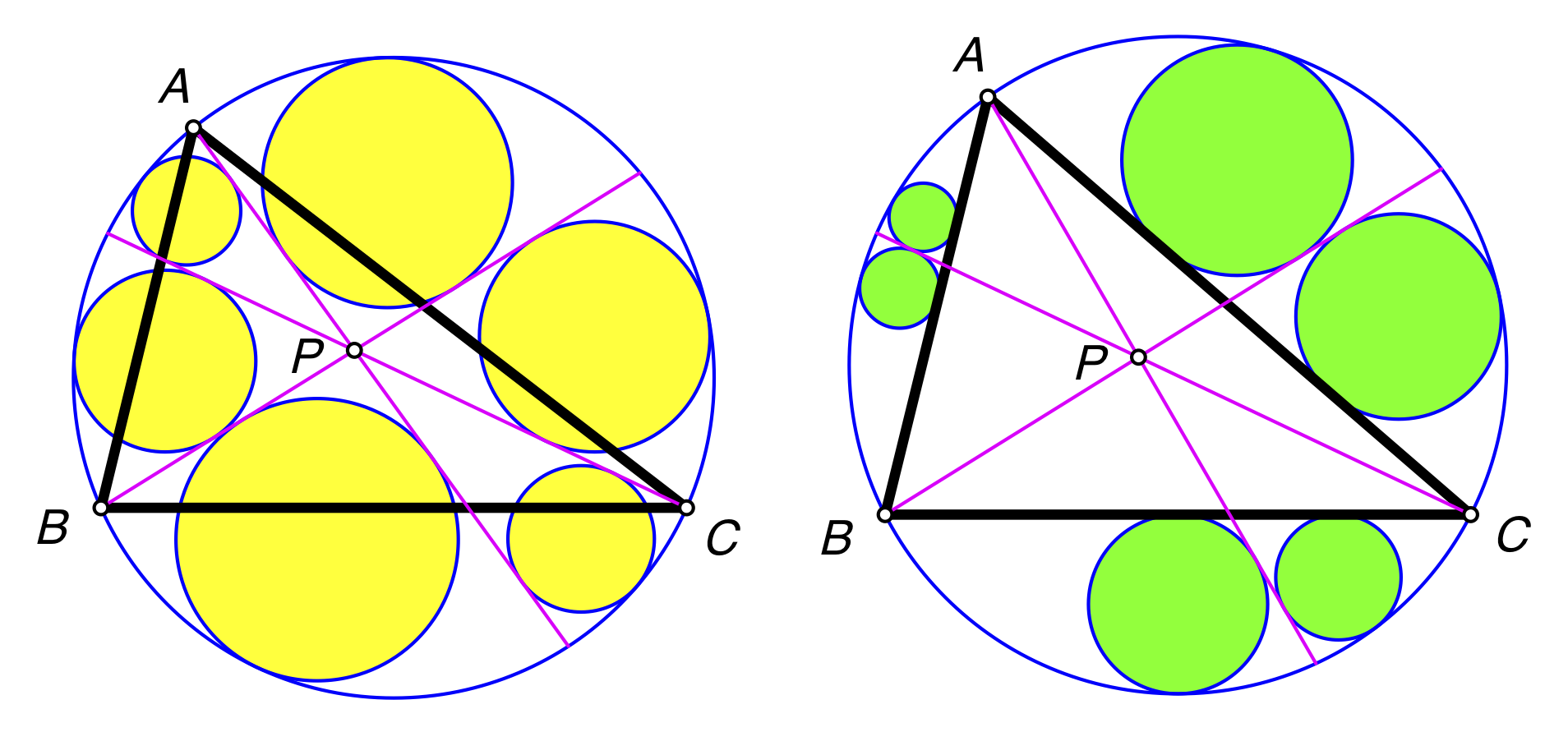}
\caption{sets of incircles tangent to circumcircle of $\triangle ABC$}
\label{fig:mixtilinear}
\end{figure}

\begin{open}
Investigate the relationship between the radii in each set of six incircles shown in Figure \ref{fig:mixtilinear}.
\end{open}

We can also form regions bounded by the incircle of $\triangle ABC$.
Figure \ref{fig:mixtilinearsIn} shows some examples.
Each red circle associated with the triangle on the left is the incircle of a region formed by two cevians and the incircle of $\triangle ABC$.
Each blue circle associated with the triangle on the right is the incircle of a region formed by one side of the triangle, one cevian, and the incircle of $\triangle ABC$.

\begin{figure}[h!]
\centering
\includegraphics[width=0.8\linewidth]{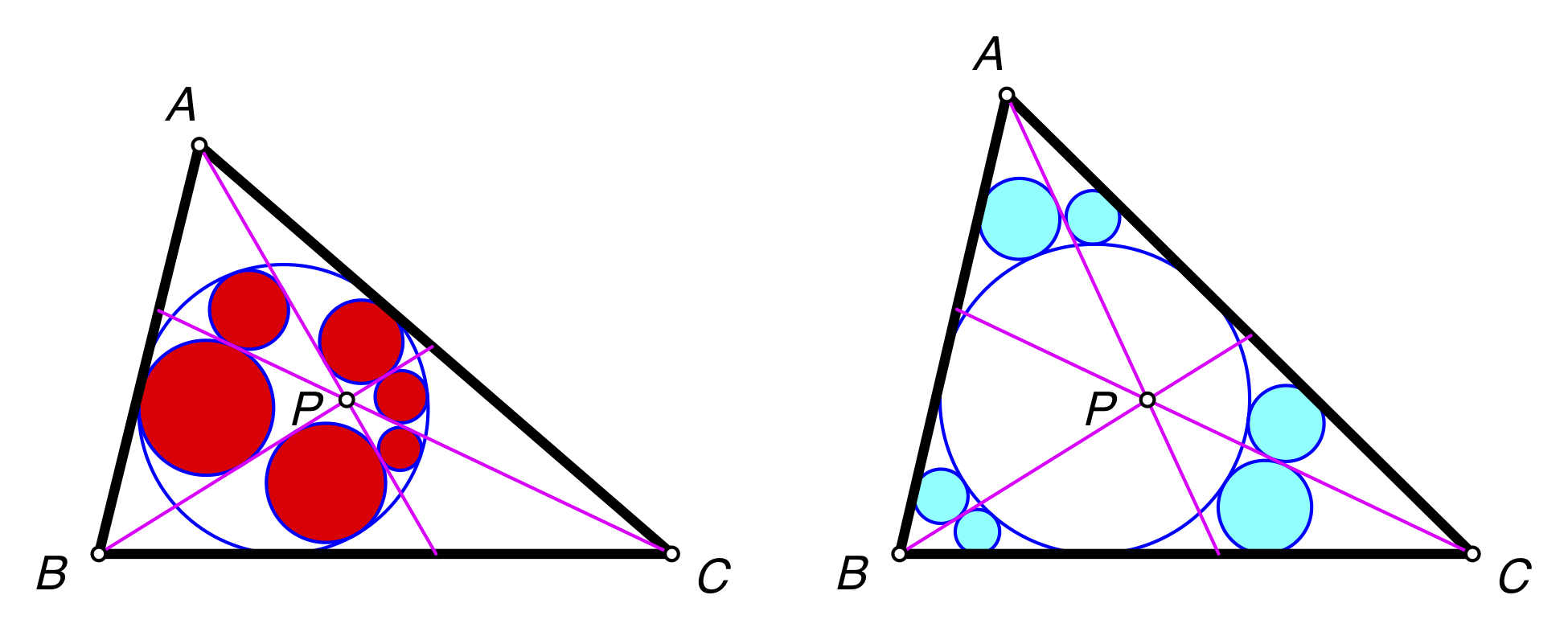}
\caption{sets of incircles tangent to incircle of $\triangle ABC$}
\label{fig:mixtilinearsIn}
\end{figure}

\begin{open}
Investigate the relationship between the radii in each set of six incircles shown in Figure \ref{fig:mixtilinearsIn}.
\end{open}

\begin{open}
For a fixed triangle $ABC$, characterize those points $P$ such that $r_1+r_3+r_5=r_2+r_4+r_6$.
\end{open}
There are many such points. Some of them look like they lie on a straight line.

\medskip
We have found relationships between the $r_i$ when $P$ is the orthocenter, circumcenter,
centroid, and incenter.

\begin{open}
Investigate the relationship between the $r_i$ when $P$ is some other notable point, such as the nine-point center, the Nagel Point, or the Gergonne Point.
\end{open}


\end{document}